\documentclass[12pt,leqno]{amsart}
\usepackage{amssymb,amsthm,amsmath,latexsym}
\usepackage{soul}
\usepackage[all]{xy}

\newtheorem{theorem}{\sc Theorem}[section]

\newtheorem{lem}[theorem]{\sc Lemma}

\newtheorem{cor}[theorem]{\sc Corollary}

\newtheorem{rem}[theorem]{\sc Remark}

\newtheorem*{thmA}{Theorem A}
\newtheorem*{thmB}{Theorem B}
\newtheorem*{thmC}{Theorem C}
\newtheorem*{thmD}{Theorem D}

\title[Weak Commutativity]{Finiteness conditions for the weak commutativity construction}
\author[Bastos]{R. Bastos}
\address{ Departamento de Matem\'atica, Universidade de Bras\'ilia,
Brasilia-DF, 70910-900 Brazil }
\email{(Bastos) bastos@mat.unb.br}
\author[Lima]{B. Lima}
\address{ Departamento de \'Areas Acad\^emicas, Instituto Federal de Goi\'as, \'Aguas Lindas-GO, 72910-733 Brazil}
\email{(Lima) bruno.cesar@ifg.edu.br}
\author[Nunes]{R. Nunes}
\address{ Departamento de Matem\'atica, Universidade Federal de Goi\'as,
Goi\^ania-GO, 74690-900 Brazil }
\email{(Nunes) ricardo@ufg.br}

\begin{document}

\maketitle

\begin{abstract}
\hspace{0,08cm} The operator, $\chi $, of weak commutativity between isomorphic groups $G$ and $G^{\varphi }$ was introduced by Sidki as \begin{equation*} \chi (G)=\left\langle G \cup G^{\varphi }\mid \lbrack g,g^{\varphi }]=1\,\forall \,g\in G\right\rangle \text{.} \end{equation*} 
It is known that the operator $\chi $ preserves group properties such as finiteness, solubility and also nilpotency for finitely generated groups. We prove that if $G$ is a locally finite group with $exp(G)=n$, then $\chi(G)$ is locally finite and has finite $n$-bounded exponent. Further, we examine some finiteness criteria for the subgroup $D(G) = \langle [g_1,g_2^{\varphi}] \mid g_i \in G\rangle \leqslant \chi(G)$ in terms of the set $\{[g_1,g_2^{\varphi}] \mid g_i \in G\}$. 
\end{abstract}

\maketitle

\section{Introduction}

Let  $G^{\varphi}$ be
a copy of the group $G$, isomorphic via $\varphi : G \rightarrow
G^{\varphi}$, given by $g \mapsto g^{\varphi}$. The following group construction was introduced
and analyzed in \cite{Sidki} $$ \chi(G) = \langle G \cup G^{\varphi} \mid [g,g^{\varphi}]=1, \forall g \in G \rangle.$$ The weak commutativity group $\chi (G)$ maps onto $G$ by $g\mapsto g$, $g^{\varphi }\mapsto g$
with kernel $L(G)=\left\langle g^{-1}g^{\varphi } \mid g\;\in G\right\rangle $ and
maps onto $G\times G$ by $g\mapsto \left( g,1\right) ,g^{\varphi
}\mapsto \left( 1,g\right) $ with kernel $D(G)=[G,G^{\varphi }]$. It is an
important fact that $L(G)$ and $D(G)$ commute. Define $T(G)$ to be the
subgroup of $G\times G\times G$ generated by $\{(g,g,1),(1,g,g)\mid g\in G\}$. Then $\chi (G)$ maps onto $T(G)$ by $g\mapsto \left( g,g,1\right)$, $g^{\varphi }\mapsto \left( 1,g,g\right) $, with kernel $W(G)=L(G)\cap D(G)$, an abelian group. In particular, the quotient $\chi(G)/W(G)$ is isomorphic to a subgroup of $G \times G \times G$. A further normal subgroup of $\chi (G)$ is $R(G)={%
[G,L(G),G^{\varphi }]}$. The quotient $W(G)/R(H)$ is isomorphic to
the Schur Multiplier $M(G)$ (cf. \cite{Roc82}).

In \cite{Sidki}, Sidki proved that if $G$ is finite, then so is $\chi(G)$. Other finiteness conditions for the weak commutativity group were considered in  \cite{GRS,LO,Roc82}. An interesting construction related to the group $\chi(G)$ was introduced by Rocco \cite{NR1}. More precisely, in \cite{NR1}, Rocco defined the group $\nu(G)$ as 
$$\begin{array}{ll} {\nu}(G) =  \langle
G \cup G^{\varphi}\ |  &
[g_1,{g_2}^{\varphi}]^{g_3}=[{g_1}^{g_3},({g_2}^{g_3})^{\varphi}] = 
[g_1,{g_2}^{\varphi}]^{g^{\varphi}_3}, \ g_i \in G
\rangle . \end{array}$$

It is a well known fact (see \cite[Proposition 2.6]{NR1}) that the subgroup
$[G, G^{\varphi}]$ of $\nu(G)$ is canonically isomorphic with the {\em non-abelian
tensor square} $G \otimes G$, as defined by Brown and Loday in their seminal paper \cite{BL}, the isomorphism being induced by $g \otimes h \mapsto
[g, h^{\varphi}]$. Other interesting homological functors appeared as sections of the non-abelian tensor square (cf. \cite[Proposition 4.10]{BL}). Moreover, in \cite{NR1,NR2}, it was proved that the constructions $\chi(G)$ and $\nu(G)$ have isomorphic quotients. More precisely, $$\dfrac{\nu(G)}{\Delta(G)} \cong \dfrac{\chi(G)}{R(G)},$$ 
where $\Delta(G) = \langle [g,g^{\varphi}] \mid g \in G\rangle \leqslant \nu(G)$. See \cite[Remark 2]{NR1} and \cite[Remark 4]{NR2}) for more details. In order to avoid confusion with other mentioned constructions, unless otherwise indicated, the only subgroups considered are related to the construction $\chi(G)$. We consider the following sets/subgroups of $\chi(G)$: $T_{\chi}(G) = \{[g,h^{\varphi}] \mid g,h \in G\}$, $D(G) = [G,G^{\varphi}]$, $L(G) = \langle g^{-1}g^{\varphi} \mid g \in G\rangle$ and $R(G) = [G,L(G),G^{\varphi}]$. 

A celebrated result due to Zelmanov \cite{ze1,ze2,ze16} is the positive solution of the Restricted Burnside Problem (RBP for short): every $m$-generator finite group of exponent $n$ has $\{m,n\}$-bounded order. As usual, the expression ``$\{a,b,...\}$-bounded'' means ``bounded from above by some function which depends only on parameters $a,b,...$''. In \cite{M}, Moravec proved that if $G$ is locally finite with exponent $\exp(G)=e$, then the group $\nu(G)$ is locally finite with $e$-bounded exponent. The group $G$ is said to have a certain property locally if each finitely generated subgroup of $G$ has this property. We establish the following related result.

\begin{thmA}
Let $n$ be a positive integer. Let $G$ be a locally finite group with $\exp(G)=n$. Then the group $\chi(G)$ is locally finite with exponent $\exp(\chi(G))$ finite and  $n$-bounded.
\end{thmA}

In \cite{BNR,BR1,BR2,BRV}, the authors study the influence of the set of all tensors $T_{\otimes}(G) = \{g \otimes h \mid g,h \in G\} \subseteq \nu(G)$ on the structure of the non-abelian tensor square $G \otimes G$ and related constructions. For instance, in \cite{BNR}, it was proved that the set of tensors $\{g \otimes h \mid g,h \in G\}$ is finite if and only if the non-abelian tensor square $G \otimes G$ is finite. In the same direction, in \cite{BNR}, it was shown that if $G$ is a finitely generated residually finite group such that the exponent of the non-abelian tensor square  $\exp(G \otimes G)$ is finite, then $G$ is finite.  We can extend these results to the context of the weak commutativity construction.    

\begin{thmB}
Let $G$ be a group. The set $T_{\chi}(G)$ is finite if and only if $D(G)$ is finite.
\end{thmB}

An immediate consequence of \cite{shu05}, is that if $G$ is a residually finite group satisfying some non-trivial identity in which all commutators $[x,y]$ have order dividing some $p$-power order, for some fixed prime $p$, then the derived subgroup $G'$ is locally finite. We obtain the following related result.  

\begin{thmC}
Let $p$ be a prime. Let $G$ be a residually finite group satisfying some non-trivial identity. Assume that for every $x,y \in G$ there exists a $p$-power $q=q(x,y)$ such that $[x,y^{\varphi}]^q=1$. Then the subgroup $D(G)$ is locally finite.
\end{thmC}

A natural question arising in the context of Theorem C is whether the
theorem remains valid with $q$ allowed to be an arbitrary natural number rather than
$p$-power. This is related to the conjecture that if $G$ is a residually finite group and every commutator $[x,y]$ has order dividing a fixed number $n$, then the derived subgroup $G'$ is locally finite (cf. \cite{shu00}).  

Given a group $G$, an element $g\in G$ is called a (left) Engel element if for any $x\in G$ there exists a positive integer $n=n(x,g)$ such that $[x,_n g]=1$, where the commutator $[x,_n g]$ is defined inductively by the rules $$[x,_1 g]=[x,g]=x^{-1}g^{-1}xg\quad {\rm and,\, for}\; n\geq 2,\quad [x,_n g]=[[x,_{n-1} g],g].$$
If $n$ can be chosen independently of $x$, then $g$ is called a (left) $n$-Engel element, or more generally a bounded (left) Engel element. The group $G$ is an Engel group (resp. an $n$-Engel group) if all its elements are Engel (resp. $n$-Engel). Following Zelmanov's solution of the RBP \cite{ze1,ze2,ze16}, Wilson proved that every $n$-Engel residually finite group is locally nilpotent \cite{wi91}. Later, Shumyatsky  shown that if $G$ is a  residually finite group in which all commutators $[x,y]$ are $n$-Engel, then the derived subgroup $G'$ is locally nilpotent \cite{shu00}. In the context of the weak commutativity construction, Gupta, Rocco and Sidki proved that if $G$ is locally nilpotent, then so is $\chi(G)$ (cf. \cite{GRS}). We obtain the following related result.

\begin{thmD}
Let $G$ be a residually finite group satisfying some non-trivial identity. Assume that for every $x,y \in G$ there exists a positive integer $n=n(x,y)$ such that $[x,y^{\varphi}]$ is $n$-Engel in $\chi(G)$. Then $D(G)$ is locally nilpotent.  
\end{thmD}

{\noindent} The paper is organized as follows. In the next section we proof Theorems A and B. Section 3 contains the proofs of Theorems C and D.  The proofs of the main results rely on Zelmanov's techniques that led to the positive solution of the RBP \cite{ze1,ze2,ze16} (see \cite{shu00} for a survey).    

\section{Local finiteness criteria for the weak commutativity construction}

Schur \cite[10.1.4]{Rob} shown that if $G$ is a group whose center $Z(G)$ has finite index $n$, then the order of the derived subgroup $G'$ is finite and the exponent $\exp(G')$ divides $n$. In particular, the group $G$ is a BFC-group. Neumann \cite[14.5.11]{Rob} improved Schur's theorem in a certain way, showing that the group $G$ is a BFC-group if and only if the derived subgroup $G'$ is finite, and that this occurs if and only if $G$ contains only finitely many commutators. Recall that a group $G$ is called a BFC-group if there is a positive integer $d$ such that no element of $G$ has more than $d$ conjugates. An immediate consequence of Schur's Theorem that if $G$ is a group in  which the quotient  $G/Z(G)$ is locally finite, then the derived subgroup $G'$ is also locally finite. In \cite{Mann},  Mann shows the following quantitative version of the above result. 

\begin{lem} \label{lem.Mann} (Mann, \cite[Theorem 1]{Mann})
Let $n$ be a positive integer. Let $G$ be a group in which $G/Z(G)$ is locally finite with $\exp(G/Z(G))=n$. Then the derived subgroup $G'$ is locally finite and $\exp(G')$ is finite with $n$-bounded exponent. 
\end{lem}

For the reader's convenience we restate Theorem A. 

\begin{thmA}
Let $n$ be a positive integer. Let $G$ be a locally finite group with $\exp(G)=n$. Then the group $\chi(G)$ is locally finite with exponent $\exp(\chi(G))$ finite and  $n$-bounded.
\end{thmA}

\begin{proof}
It is well know that the quotient $\chi(G)/W(G)$ is isomorphic to a subgroup of $G \times G \times G$. So the groups $\chi(G)/W(G)$ and $L(G)/W(G)$ are locally finite with exponent $n$. Note that $W(G)$ is a central subgroup of $L(G)$. By Lemma \ref{lem.Mann}, we deduce that the derived subgroup  $L(G)'$ is locally finite with $n$-bounded exponent. There is no loss of generality in assuming that $L(G)$ is abelian. Since $L(G)$ is generated by elements of the form $[g,\varphi] = g^{-1}g^{\varphi}$, $g \in G$, it is sufficient to prove that these generators have $n$-bounded orders. Note that $[g^k,\varphi] = [g,\varphi]^k$, for all $k \in \mathbb{Z}$. Let $m$ be the order of the element $g \in G$.  Consequently,  $$1=[g^m,\varphi]=[g,\varphi]^m.$$ In particular, the order of the element $[g,\varphi]$ divides $n$ for every $g \in G$. We deduce that $L(G)$ is locally finite and $\exp(L(G))$ is finite with $n$-bounded exponent, as well. The proof is complete. 
\end{proof}

\begin{rem}
In the above result we use Mann's theorem \cite{Mann}, which relies on the positive solution of the RBP.
\end{rem}

The following result is an immediate consequence of \cite[Theorem C (iii)]{Sidki}.

\begin{lem}
\label{lem.finite.Sidki} Let $G$ be a finite group. Then the weak commutativity group $\chi(G)$ is finite. 
\end{lem}

We are now in a position to prove Theorem B. 

\begin{thmB}
Let $G$ be a group. The set $T_{\chi}(G)$ is finite if and only if $D(G)$ is finite.
\end{thmB}
\begin{proof}
Clearly, if $D(G)$ is finite, then $T_{\chi}(G)$ is finite. So we only need to prove the converse.  

Since $T_{\chi}(G)$ is finite, it follows that the set of all commutators is finite. In particular, the group $G$ is a BFC-group. By Neumann's result \cite[14.5.11]{Rob}, the derived subgroup $G'$ is finite. As $W(G)$ is a central subgroup of $D(G)$ and the quotient group $D(G)/W(G)$ is isomorphic to the derived subgroup $G'$ we have $W(G)$ is a central subgroup of finite index in $D(G)$. Without loss of generality we may assume that $D(G)$ is abelian. By \cite[4.1.13]{Sidki} we have a the following exact sequence
$$ [G',G^{\varphi}]\hookrightarrow [G,G^{\varphi}]\twoheadrightarrow \left[G^{ab},\left(G^{ab}\right)^{\varphi} \right].  $$

By \cite[4.1.13]{Sidki},  
$$[G',G^{\varphi}] = [G,G^{\varphi}] \cap \left<G', \left(G'\right)^{\varphi} \right>,$$
which is finite, because the derived subgroup $G'$ is finite and the subgroup $\langle G',(G')^{\varphi} \rangle$ is an epimorphic image of $\chi(G')$ (Lemma \ref{lem.finite.Sidki}) and we may assume that $G$ is abelian. Hence, for all $a,b \in G$, 
$$[a^2,b^\varphi] = [a,b^\varphi]^2 \in T_{\chi}(G). $$
Since $T_{\chi}(G)$ is finite, it follows that every element $[a,b^{\varphi}]$ has finite order. We conclude that the subgroup $D(G)$ is  finite, which completes the proof. 
\end{proof}

\section{Finiteness conditions for the weak commutatity of residually finite groups}

Recall that a group $G$ is called an FC-group if every element of $G$ has a finite number of conjugates. A subset $X$ of a group is commutator-closed if $[x,y]\in X$ for any $x,y\in X$. We need the following result, due to Shumyatsky \cite{shu05}. 

\begin{lem} \label{lem.shu} 
Let $G$ be a residually finite group satisfying some non-trivial identity $f \equiv~1$. Suppose $G$ is generated by a normal commutator-closed set $X$ of $p$-elements. Then $G$ is locally finite.
\end{lem}

We are now in a position to prove Theorem C. 
\begin{proof}[Proof of Theorem C]
We first prove that the derived subgroup $G' = \langle [x,y] \mid x,y \in G\rangle$ is locally finite. For every $x,y \in G$ there exists a $p$-power $q=q(x,y)$ such that $[x,y^{\varphi}]^q=1$. In particular, we deduce that every commutator has finite $p$-power order. By Lemma \ref{lem.shu}, the derived subgroup $G'$ is locally finite.

Let $M$ be a finitely generated subgroup of $D(G)$. Clearly, there exist finitely many elements $a_1,\ldots,a_s$, $b_1, \ldots,b_s \in G$ such that $$M  \leqslant \langle [a_i,b_i^{\varphi}]  \mid \ i = 1,\ldots,s  \rangle = N.$$
It suffices to prove that $N$ is finite. Since the subgroup $W(G)$ is central in $D(G)$ and the factor group $D(G)/W(G)$ is isomorphic to $G'$, it follows that $N$ is a central-by-finite group. By Schur's Theorem \cite[10.1.4]{Rob}, the derived subgroup $N'$ is finite, so $N$ is an FC-group. Since the torsion set form a subgroup in FC-groups (Neumann, \cite[14.5.9]{Rob}), we deduce that $N$ is finite. Since $M$ was chosen arbitrarily, we now conclude that $D(G)$ is locally finite. The proof is complete.
\end{proof}

\begin{cor}
Let $m$ be a positive integer and $p$ a prime. Let $G$ be a residually finite group. Suppose that for every $x,y \in G$ the element $[x,y^{\varphi}]$ has order dividing $p^m$. Then $D(G)$ is locally finite. 
\end{cor}

\begin{proof}
We first show that the group $G$ satisfies a non-trivial identity. For every $x,y \in G$ the element $[x,y^{\varphi}]$ has order dividing $p^m$. In particular, we deduce that every commutator $[x,y]$ has order dividing $p^m$ and so, the group $G$ satisfies the identity $$ f=[x,y]^{p^{m}}\equiv1.$$  
Applying Theorem C to $D(G)$, we deduce that $D(G)$ is locally finite, as well. \end{proof}

\begin{rem}
Note that the finiteness of $D(G)$ does not imply the finiteness of the group $G$. For instance, if $G = C_{\infty}$, then $D(G)$ is trivial and $\chi(G) \cong C_{\infty} \times C_{\infty}$.
\end{rem}

The next lemma is taken from \cite{BMTT}. 

\begin{lem} (\cite[Theorem A]{BMTT}) \label{lem.BMTT}
Let $G$ be a residually finite group satisfying a non-trivial identity. Suppose that $G$ is generated by a commutator-closed set $X$ of bounded Engel elements. Then $G$ is locally nilpotent.
\end{lem}

\begin{thmD}
Let $G$ be a residually finite group satisfying some non-trivial identity. Assume that for every $x,y \in G$ the element  $[x,y^{\varphi}]$ is a bounded Engel element in $\chi(G)$. Then $D(G)$ is locally nilpotent.  
\end{thmD}

\begin{proof}
Since $W(G)$ is a central subgroup of $D(G)$ and $D(G)/W(G)$ is isomorphic to the derived subgroup $G'$, it suffices to prove that $G'$ is locally nilpotent.  

For every $x,y \in G$ there exists a positive integer $n=n(x,y)$ such that the element $[x,y^{\varphi}]$ is $n$-Engel in $\chi(G)$. In particular, we deduce that for every $x,y \in G$ the commutator $[x,y]$ is a bounded Engel (in $G$). By Lemma \ref{lem.BMTT}, the derived subgroup $G'$ is locally nilpotent. The proof is complete. 
\end{proof}

\begin{cor}
Let $n$ be a positive integer. Let $G$ be a residually finite group. Assume that for every $x,y \in G$ the element $[x,y^{\varphi}]$ is $n$-Engel (in $\chi(G)$). Then $D(G)$ is locally nilpotent.   
\end{cor}

\begin{proof}
We first prove that $G$ satisfies a non-trivial identity. Since every element $[x,y^{\varphi}]$ is $n$-Engel (in $\chi(G)$), it follows that every commutator $[x,y]$ is $n$-Engel (in $G$). In particular, the group $G$ satisfies the identity $$f=[z,_{n}[x,y]] \equiv 1.$$ Applying Theorem D for $D(G)$, we obtain that $D(G)$ is locally nilpotent. The proof is complete.   
\end{proof}

\noindent {\bf Acknowledgments.} We thank N.\,R. Rocco for collaborating with us on this project. This work was partially supported by FAPDF - Brazil, Grant: 0193.001344/2016.


\begin{thebibliography}{10}

\bibitem{BMTT} R. Bastos, N. Mansuro\u{g}lu, A. Tortora, M. Tota, {\it Bounded Engel elements in groups satisfying an identity}, Arc. Math., {\bf 110} (2018) pp. 311--318.

\bibitem{BNR} R. Bastos, I.\,N. Nakaoka and N.\,R. Rocco, { \it Finiteness conditions for the non-abelian tensor product of groups}, Monatsh. Math., {\bf 187} (2018) pp. 603--615. 

\bibitem{BR1} R. Bastos and N.\,R. Rocco, {\it The non-abelian tensor square of residually finite groups}, Monatsh. Math., {\bf 183} (2017) pp. 61--69.   

\bibitem{BR2} R. Bastos and N.\,R. Rocco, {\it Non-abelian tensor product of residually finite groups}, S\~ao Paulo J. Math. Sci., {\bf 11} (2017) pp. 361--369.

\bibitem{BRV} R. Bastos, N.\,R. Rocco and E.\,R. Vieira,  {\it Finiteness of homotopy groups related to the non-abelian tensor product},  	arXiv:1812.07559 [math.GR].  

\bibitem{BL} R. Brown, and J.-L. Loday, {\it Van Kampen theorems for diagrams of spaces}, Topology, {\bf 26} (1987) pp. 311--335.

\bibitem{GRS} N. Gupta, N. Rocco and S. Sidki, {\it Diagonal embeddings of nilpotent groups}, {\bf 30 } (1986)  pp. 274--283. 

\bibitem{LO} B.\,C.\,R. Lima and R.\,N. Oliveira, {\it Weak commutativity between two isomorphic polycyclic groups}, J. Group Theory {\bf 19} (2016) pp. 239--248. 

\bibitem{Mann} A. Mann, {\it The exponents of central factor and commutator groups}, J. Group Theory, {\bf 10} (2007) pp. 435--436.

\bibitem{M} P. Moravec, {\em The exponents of nonabelian tensor products of groups}, J. Pure Appl. Algebra, {\bf 212} (2008)  pp. 1840--1848.

\bibitem{Rob} D.\,J.\,S. Robinson,
\textit{A course in the theory of groups}, 2nd edition, Springer-Verlag, New York, 1996.

\bibitem{Roc82} N.\,R. Rocco, {\it On weak commutativity between finite p-groups, p odd}, J. Algebra {\bf 76} (1982) 471--488

\bibitem{NR1} N.\,R. Rocco, {\it On a construction related to the non-abelian 
tensor square of a group}, Bol. Soc. Brasil Mat., {\bf 22} (1991) pp. 63--79.

\bibitem{NR2} N.\,R. Rocco, {\it A presentation for a crossed embedding of finite solvable groups}, Comm. Algebra {\bf 22} (1994) pp. 1975--1998.

\bibitem{shu00} P. Shumyatsky,
Applications of Lie ring methods to group theory, in {\it Nonassociative algebra and its applications}, eds. R. Costa, A. Grishkov, H. Guzzo Jr.
and L.\,A. Peresi,
Lecture Notes in Pure and Appl. Math., Vol. 211 (Dekker, New York, 2000) pp. 373--395.

\bibitem{shu05} P. Shumyatsky, {\it Elements of prime power order in residually finite groups}, Int. J. Algebra Comput. {\bf 15} (2005) pp. 571--576. 

\bibitem{Sidki} S.\,N. Sidki, {\it On weak permutability between groups}, J. Algebra, {\bf 63}, (1980) pp. 186--225.

\bibitem{wi91} J.\,S. Wilson, {\it Two-generator conditions for residually finite
groups}, Bull. London Math. Soc. {\bf 23} (1991) pp. 239--248.

\bibitem{ze1} E. Zelmanov, 
{\it The solution of the restricted Burnside problem for groups of odd exponent}, Math. USSR Izv., {\bf 36} (1991) pp. 41--60.

\bibitem{ze2} E. Zelmanov,
{\it The solution of the restricted Burnside problem for 2-groups}, Math. Sb., {\bf 182} (1991) pp. 568--592.

\bibitem{ze16} E.\,I. Zelmanov, {\it Lie algebras and torsion groups with identity}, J. Comb. Algebra, {\bf 1}, (2017) pp. 289--340.

\end{thebibliography}
\end{document}